\documentclass[11pt]{article}  
\usepackage{amssymb,amsmath, amsthm,amsfonts}

\usepackage{graphicx}

\usepackage{subfigure}
\usepackage{epstopdf}                           
\def\disp{\displaystyle}
\def\dref#1{(\ref{#1})}
\def\dfrac{\displaystyle\frac}
\newtheorem{theorem}{Theorem}[section]

\newtheorem{proposition}{Proposition}[section]
\newtheorem{lemma}{Lemma}[section]

\newtheorem{remark}{Remark}[section]
\newtheorem{example}{Example}[section]
\numberwithin{equation}{section}

\begin{document}

	\title{\bf  Boundary control of cascaded ODE-Heat equations under actuator saturation 
		\footnote{\small This work was supported by Israel Science Foundation (grant No
			1128/14).}}
	\author{Wen Kang\footnote{\small  Email:
			kangwen@amss.ac.cn}  \hspace{0.5cm}  Emilia Fridman
		\footnote{\small Email: 
			emilia@eng.tau.ac.il}\\
		{\it Department of Electrical Engineering-Systems, Tel Aviv University, Israel}}
	
	\maketitle

%

\begin{abstract}
		In this paper, we consider boundary stabilization for  a cascade of
		ODE-heat system with a time-varying state delay
		under actuator saturation.
		To stabilize the system, we design a state feedback controller via the backstepping  method
		and find a bound on the domain of attraction. The latter bound is based on Lyapunov method,
		whereas the exponential stability conditions for the delayed cascaded system are derived by using Halanay's inequality.
		Numerical examples illustrate the efficiency of the method.
			\end{abstract}
	
			\section{Introduction}
			     In the last few years, coupled systems have attracted considerable attention in research communities.
			Stabilization of the cascade of PDE systems was dealt with in \cite{orlov,Tsubakinto}. Controller design for PDE-ODE cascade systems has been extensively studied for many types of coupling such as ODE-Reaction diffusion equation (see e.g. \cite{M.Krstic,M.Krstic2,Tang}), ODE-Wave equation (see e.g. \cite{M.Krstic3}), and
			ODE-Schr\"{o}dinger equation (see e.g. \cite{Renbeibei}).
			In order to stabilize the cascaded PDE-ODE systems, the backstepping method has been applied in \cite{Renbeibei,M.Krstic,M.Krstic3,M.Krstic2,Tang}. The idea is to use a Volterra integral transformation to transform the original system to a target system \cite{M.Krstic1}.\par			
		        Stabilization for systems described by PDEs  subject to time delay has received much attention in recent years.
			An effective linear matrix inequality (LMI) approach is proposed to analysis and design for time delay PDE systems in \cite{Emilia4,Emilia3,Emilia2,Emilia,Emilia1}.
			In \cite{Tomoaki}, based on the backstepping method, a control strategy for reaction-diffusion equations with a constant state delay is proposed. \par
			 For practical application of backstepping controllers, in many cases the constraints on the control input should be taken into account. There have been some important results about PDEs subject to distributed control constraints (see e.g. \cite{ElFarra,Prieur2,Prieur}). However,
			boundary control of PDEs in the presence of actuator saturation has not been studied yet in the literature.
			In the present paper we introduce stabilizing backstepping-based boundary controllers
			for coupled heat-ODE systems with time-varying state delays in the presence of actuator saturation.
			We first extend the backstepping method to the latter class of delayed systems.
			Differently from the non-delayed case, the resulting target heat equation  is coupled with the ODE system. However, each subsystem contains design parameters that allows to stabilize the coupled system.
			By using Lyapunov method for the target system, we find a bound on the domain of attraction of this system, and further on the domain of attraction of the original system.
For simplicity only, our conditions are based on delay-independent stability condition for finite-dimensional system with delay. Less conservative delay-dependent conditions can be derived by employing Lyapunov-Krasovskii functionals similar to  \cite{Emilia5,Tarbouriech}.\par			
		        The structure of the paper is as follows.  In the next section, the problem statement is presented and the backstepping transformation is introduced. Based on the backstepping method, a state feedback  boundary controller to the original system is designed. Section
			3 is devoted to  the existence and uniqueness of the solution for the closed-loop system with state delay. In Section 4,  delay-independent LMI conditions are presented for the stability analysis of the target system. In Section 5, we design a controller under actuator saturation via LMIs. We find an estimate on the set of initial conditions (as large as we can get) starting from which the state trajectories of the system are exponentially converging to zero. Examples with numerical simulations are presented in Section 6 for illustration of the effectiveness of the method. Some concluding remarks are presented in Section 7.
			
			{\bf Notation}.
			Throughout the paper, the superscript `$\top$' stands for matrix transposition, $\mathbb{R}^{n}$ denotes the n-dimensional Euclidean space with the norm $|\cdot|$, $L^2(0,1)$ stands for the Hilbert space of square integrable scalar functions  on $(0,1)$ with the corresponding norm $\|\cdot\|$. The notation $P>0$ denotes that $P$ is symmetric and positive definite.  For any $U$ we denote by ${\rm sat}(U,\bar{u})={\rm sign}(U)\min(|U|,\bar{u})$. Given a Banach space $\mathcal H$, the space of the continuous $\mathcal{H}$-valued functions $z:[a,b] \rightarrow \mathcal{H}$ with the induced norm
			$\|z\|_{C([a,b], \mathcal {H})}=\max\limits_{s\in[a,b]}\|z(s)\|_{\mathcal H} $ is denoted by $C([a,b], \mathcal H)$.
			
			    	\section{Backstepping control for cascaded ODE-Heat equations with delay}	
			    	In this section, we consider the following coupled ODE-reaction diffusion system:
			    	\begin{equation}\label{an}
			    	\left\{\begin{array}{ll}
			    	\dot{X}(t)=AX(t)+A_1X(t-\tau(t))+Bu(0,t),\\
			    	u_t(x,t)= u_{xx}(x,t)+a_2u(x,t-\tau(t))+au(x,t),\\
			    	u_x(0,t)=0,\\
			    	(X(t),u(x,t))=(f(t),\psi(x,t)),\; -h\le t\le 0,
			    	\end{array}\right.
			    	\end{equation}
			    	with Dirichlet boundary actuator: \begin{equation}\label{dir}
			    	u(1,t)=U(t),\;t>0,
			    	\end{equation}
			    	or Neumann boundary actuator:
			    	\begin{equation}\label{neu}
			    	u_x(1,t)=U(t),\;t>0.
			    	\end{equation}
			    		Here $x\in (0,1)$, $A, A_1 \in \mathbb{R}^{n\times n}$, $B \in \mathbb{R}^{n\times 1}$, $a, a_2\in \mathbb{R}$ denotes a constant coefficient, $\tau(t)$ corresponds to a time varying delay, and
			    		$(f(t),\psi(x,t))$ is the initial state defined for $0\le x\le 1,\; -h\le t\le 0$.
			    		$X(t) \in \mathbb{R}^{n}$ is the state of ordinary differential equation, $u(x, t)\in \mathbb{R}$ is the displacement of heat equation, and $U(t) \in \mathbb{R}$ is the control actuation.
			    		
			    		  We assume that $(A,B)$ is controllable. 	Assume that the time-varying delay $\tau(t)$ is a continuously differentiable function of $t$ that satisfies
			    		\begin{equation}\label{ta}
			    	       0< h_0\le \tau(t)\le h
			    	      \end{equation}
			    	      with some constants $h_0$ and $h>0$. Note that the assumption $h_0>0$ is used for simplification of the proof of well-posedness.  The delay and its bounds may be unknown for the exponential stability conditions (without finding a decay rate) and for the domain of attraction in the presence of actuator saturation. However, the upper bound $h$ on the delay should be known for finding  a bound on the decay rate of the exponential stability.

			    	
			    		
                     	The first equation of \dref{an} is ODE with delay or a difference-differential equation.
                      	So, we call it ODE in order to distinguish it from PDE.
			    		First, we look for a coordinate transformation
			    		\begin{equation}\label{wei}
			    		\left\{\begin{array}{ll}
			    		X(t)=X(t),\\
			    		w(x,t)=\disp u(x,t)-\int_0^xk(x,y)u(y,t)dy-\gamma(x)X(t),
			    		\end{array}\right.
			    		\end{equation}
			    		that transforms the system \dref{an} into the following
			    		intermediate ODE-heat cascade:
			    		\begin{equation}\label{bn}
			    		\left\{\begin{array}{ll}
			    		\dot{X}(t)=(A+BK)X(t)+A_1X(t-\tau(t))+Bw(0,t),\\
			    		w_t(x,t)= w_{xx}(x,t)+a_2w(x,t-\tau(t))+aw(x,t)-\gamma(x)[A_1-a_2I]X(t-\tau(t)),\\
			    		w_x(0,t)=0,\\
			    		(X(t),w(x,t))=(f(t),\phi(x,t)),\; -h\le t\le 0,
			    		\end{array}\right.
			    		\end{equation}
			    		where $K$ is chosen such that 
			    				    		$$\dot{X}(t)=(A+BK)X(t)+A_1X(t-\tau(t))$$
			    		 is asymptotically stable,
			    		and
			    		\begin{equation}
			    		\phi(x,t)=\psi(x,t)-\int_0^xk(x,y)\psi(y,t)dy-\gamma(x)f(t).
			    		\end{equation}
			    		Boundary actuation \dref{dir} is transformed into
			    		\begin{equation}\label{yi}
			    		w(1,t)=\disp U(t)-\int_0^1k(1,y)u(y,t)dy-\gamma(1)X(t),
			    		\end{equation}
			    		and \dref{neu} is transformed into
			    		\begin{equation}\label{yin}
			    		\begin{array}{ll}
			    		w_x(1,t)=\disp U(t)-k(1,1)u(1,t)\disp-\int_0^1k_x(1,y)u(y,t)dy
			    		-\gamma^\prime(1)X(t).
			    		\end{array}
			    		\end{equation}	
			    		
			    		Second, a further transformation, where $(X,w)\mapsto(X,z)$, can be given by
			    		\begin{equation}\label{lyp}
			    		\left\{\begin{array}{ll}
			    		X(t)=X(t),\\
			    		z(x,t)=\disp w(x,t)-\int_0^xq(x,y)w(y,t)dy.
			    		\end{array}\right.
			    		\end{equation}
			    		Here the kernel $q(x,y)$  should be chosen to transform the system \dref{bn} into the target ODE-heat cascade:
			    		\begin{equation}\label{cn}
			    		\left\{\begin{array}{ll}
			    		\dot{X}(t)=(A+BK)X(t)+A_1X(t-\tau(t))+Bz(0,t),\\
			    		z_t(x,t)= z_{xx}(x,t)-cz(x,t)+a_2z(x,t-\tau(t))
			    		-[\gamma(x)-\disp\int_0^xq(x,y)\gamma(y)dy](A_1-a_2I)X(t-\tau(t)),\\
			    		z_x(0,t)=0,\\
			    		(X(t),z(x,t))=(f(t),\varphi(x,t)),\; -h\le t\le 0,
			    		\end{array}\right.
			    		\end{equation}
			    		where $c>0$ is a constant, and
			    		\begin{equation}
			    		\varphi(x,t)=\phi(x,t)-\int_0^xq(x,y)\phi(y,t)dy.
			    		\end{equation}
			    		Boundary actuation \dref{yi} is transformed into
			    		\begin{equation}\label{li}
			    		\begin{array}{ll}
			    		z(1,t)=\disp U(t)-\int_0^1k(1,y)u(y,t)dy-\gamma(1)X(t)\disp-\int_0^1q(1,y)w(y,t)dy,
			    		\end{array}
			    		\end{equation}
			    		and \dref{yin} is transformed into
			    		\begin{equation}\label{lil}
			    		\begin{array}{ll}
			    		z_x(1,t)&=\disp U(t)-k(1,1)u(1,t)-\int_0^1k_x(1,y)u(y,t)dy
			    		-\gamma^\prime(1)X(t)-q(1,1)w(1,t)\\&\disp-\int_0^1q_x(1,y)w(y,t)dy.
			    		\end{array}
			    		\end{equation}
			    		
			    		Next, we compute the kernels of $k(x,y)$, $\gamma(x)$ and $q(x,y)$. Motivated by \cite{Tomoaki}, we will show that the transformation for undelayed equations (see \cite {M.Krstic2}) still works for the above class of delayed equations.

			    		Differentiation of transformation \dref{wei} with respect to $t$ yields
			    			$$
			    			\begin{array}{ll}
			    			w_t(x,t)&\disp=u_{xx}(x,t)+a_2u(x,t-\tau(t))\disp+a[u(x,t)-\int_0^xk(x,y)u(y,t)dy] \\&\disp-\int_0^xk(x,y)[u_{yy}(y,t)+a_2u(y,t-\tau(t))]dy
			    		-\gamma(x)[AX(t)+A_1X(t-\tau(t))+Bu(0,t)].
			    			\end{array}
			    			$$
			    			Substitution of \dref{wei} into the resulting equation implies
			    		 $$
			    		 \begin{array}{ll}
			    		 w_t(x,t) &=u_{xx}(x,t)+a_2w(x,t-\tau(t))+aw(x,t)-k(x,x)u_x(x,t)+k_y(x,x)u(x,t)\\&-k_y(x,0)u(0,t)\disp-\int_0^xk_{yy}(x,y)u(y,t)dy-\gamma(x)Bu(0,t)\\
			    		 &-\gamma(x)[(A-aI)X(t)+(A_1-a_2I)X(t-\tau(t))].
			    		 \end{array}
			    		 $$
			    		  Similarly, the first and the second derivatives of $w(x,t)$ with respect to $x$ are given by
			    		  $$
			    		  \begin{array}{ll}
			    		  w_x(x,t)=u_x(x,t)-k(x,x)u(x,t)\disp-\int_0^xk_x(x,y)u(y,t)dy-\gamma^{\prime}(x)X(t),
			    		  \end{array}
			    		 $$
			    		  $$
			    		  \begin{array}{ll}
			    		  w_{xx}(x,t)&\disp=u_{xx}(x,t)-\dfrac{d}{dx}k(x,x)u(x,t)-k(x,x)u_x(x,t)-k_x(x,x)u(x,t)\\&\disp-\int_0^xk_{xx}(x,y)u(y,t)dy-\gamma^{\prime\prime}(x)X(t).
			    		  \end{array}
			    		  $$
			    		  Substituting \dref{wei} into \dref{an} and comparing with \dref{bn}, we
			    		  obtain the following set of conditions on the kernels
			    		  $k(x,y)$ and $\gamma(x)$
			    		  (see e.g.\cite{M.Krstic}):
			    		  \begin{equation}\label{zong}
			    		  \left\{\begin{array}{ll}
			    		  k_{xx}(x,y)=k_{yy}(x,y),\cr k_y(x,0)=-\gamma(x)B,\cr
			    		  k(x,x)=0,
			    		  \end{array}\right.
			    		  \end{equation}
			    		  and
			    		  \begin{equation}
			    		  \left\{\begin{array}{ll}\label{ty}
			    		  \gamma^{\prime\prime}(x)=\gamma(x)(A-aI),\cr \gamma(0)=K,\cr \gamma^\prime(0)=0.
			    		  \end{array}\right.
			    		  \end{equation}
			    		  The solution to \dref{zong} and \dref{ty} is given by
			    		  \begin{equation}\label{q}
			    		  \begin{array}{ll}
			    		  k(x,y)=\disp\int_0^{x-y} \gamma(\sigma)Bd\sigma,\\
			    		  \gamma(x)=\left[\begin{array}{cc} K &0 \end{array}\right]e^{\left[\begin{array}{cc} 0 &A-aI\\
			    		  	I&0 \end{array}\right]x}\left[\begin{array}{c} I\\0 \end{array}\right].
			    		  \end{array}
			    		  \end{equation}
			    		  In the similar manner, the change of variable \dref{wei} has an inverse transformation:
			    		  \begin{equation}\label{h}
			    		  u(x,t)=w(x,t)+\int_0^x n(x,y)w(y,t)dy+\psi(x)X(t),
			    		  \end{equation}
			    		  where
			    		  \begin{equation}\label{xia}
			    		  \begin{array}{ll}
			    		  n(x,y)=\disp\int_0^{x-y}\psi(\sigma)Bd\sigma,\\ \psi(x)=\left[\begin{array}{cc} K &0 \end{array}\right]e^{\left[\begin{array}{cc} 0 &A+BK-aI\\
			    		  	I&0 \end{array}\right]x}\left[\begin{array}{c} I\\0 \end{array}\right].
			    		  \end{array}
			    		  \end{equation}
			    		
			    		  By the standard procedures (see \cite{M.Krstic1}), we differentiate transformation \dref{lyp} with respect to $t$ and $x$ respectively to obtain
			    		  \begin{equation}
			    		  \begin{array}{ll}\label{j}
			    		    z_t(x,t)&\disp=w_{xx}(x,t)+a_2z(x,t-\tau(t))+aw(x,t)-q(x,x)w_x(x,t)+q_y(x,x)w(x,t)\\
			    		    &-q_y(x,0)w(0,t)\disp-\int_0^xq_{yy}(x,y)w(y,t)dy-a\int_0^xq(x,y)w(y,t)dy\\
			    		  &-[\gamma(x)-\disp\int_0^xq(x,y)\gamma(y)dy](A_1-a_2I)X(t-\tau(t)),
			    		  \end{array}
			    		  \end{equation}
			    		  \begin{equation}\label{k}
			    		  \begin{array}{ll}
			    		  z_x(x,t)=w_x(x,t)-q(x,x)w(x,t)\disp-\int_0^xq_x(x,y)w(y,t)dy,
			    		  \end{array}
			    		  \end{equation}
			    		  \begin{equation}\label{m}
			    		  \begin{array}{ll}
			    		  z_{xx}(x,t)=w_{xx}(x,t)-\dfrac{d}{dx}q(x,x)w(x,t)
			    		  -q(x,x)w_x(x,t)-q_x(x,x)w(x,t)-\disp\int_0^xq_{xx}(x,y)w(y,t)dy.
			    		  \end{array}
			    		  \end{equation}
			    		  Subtracting \dref{m} from \dref{j} and comparing with the second equation of \dref{cn}, we obtain that $q(x,y)$ satisfies
			    		
			    		  \begin{equation}\label{guo2}
			    		  \left\{\begin{array}{ll}
			    		  q_{xx}(x,y)=q_{yy}(x,y)+(a+c)q(x,y),\cr q_y(x,0)=0,\cr
			    		  q(x,x)=-\dfrac{a+c}{2}x.
			    		  \end{array}\right.
			    		  \end{equation}
			    		  The solution to \dref{guo2} is given by
			    		  $$
			    		  q(x,y)=-(a+c)x\dfrac{I_1(\sqrt{(a+c)(x^2-y^2)})}{\sqrt{(a+c)(x^2-y^2)}},
			    		  $$
			    		 where $I_1(\cdot)$ denotes the modified Bessel function of the first order: $$I_1(x)=\sum_{n=0}^{\infty}\dfrac{(x/2)^{2n+1}}{n!(n+1)!}.$$
			    		  In the similar manner, the change of variable \dref{lyp} has an inverse transformation:
			    		  \begin{equation}\label{l}
			    		  w(x,t)=z(x,t)+\int_0^xl(x,y)z(y,t)dy,
			    		  \end{equation}
			    		  where
			    		  \begin{equation}\label{n}
			    		  l(x,y)=-(a+c)x\dfrac{J_1(\sqrt{(a+c)(x^2-y^2)})}{\sqrt{(a+c)(x^2-y^2)}},
			    		  \end{equation}
			    		  where $J_1(\cdot)$ is Bessel function of the first order: $$J_1(x)=\sum_{n=0}^{\infty}\dfrac{(-1)^n(x/2)^{2n+1}}{n!(n+1)!}.$$
			    		  \subsection{Dirichlet actuation}
			    		  Next, we design the state feedback controller for the target system \dref{cn}.
			    		  By selecting the following feedback controller:
			    		  \begin{equation}\label{con}
			    		  \begin{array}{ll}
			    		  U(t)=\disp\int_0^1k(1,y)u(y,t)dy+\gamma(1)X(t)+\disp\int_0^1q(1,y)\left[ u(y,t)-\int_0^yk(y,s)u(s,t)ds-\gamma(y)X(t)\right]dy,
			    		  \end{array}
			    		  \end{equation}
			    		  one arrives to the closed-loop system of \dref{cn} with boundary actuation \dref{li} as follows:
			    		  	\begin{equation}\label{o}
			    		  	\left\{\begin{array}{ll}
			    		  	\dot{X}(t)=(A+BK)X(t)+A_1X(t-\tau(t))+Bz(0,t),\\
			    		  	z_t(x,t)= z_{xx}(x,t)-cz(x,t)+a_2z(x,t-\tau(t))
			    		  	-[\gamma(x)-\disp\int_0^xq(x,y)\gamma(y)dy](A_1-a_2I)X(t-\tau(t)),\\
			    		  	z_x(0,t)=0,\\	
			    		  	(X(t),z(x,t))=(f(t),\varphi(x,t)),\; -h\le t\le 0,
			    		  	\end{array}\right.
			    		  	\end{equation}
			    		  	subject to
			    		  	\begin{equation}\label{d}
			    		  	z(1,t)=0.
			    		  	\end{equation}
			    		  \begin{remark}
			    		  	Differently from the non-delayed case \cite{M.Krstic}, the resulting target system \dref{o}, \dref{d} is coupled. However, each differential equation (for $X$ and for $z$) contains the design parameter (either $K$ or $c$). This allows to stabilize the target system by appropriate choice of  $K$ and $c$ (see (ii) of Propositions \ref{ni}, \ref{nii} and Remark \ref{chen} below).
			    		  	\end{remark}
			    		  	\subsection{Neumann actuation}
			    		  	The Neumann controller is obtained using the same exact transformation as in the case of the Dirichlet actuation, but with the appropriate change in the boundary condition of the target system. In this case, the backstepping approach yields the following controller for the target system \dref{cn}:
			    		  	\begin{equation}\label{conc}
			    		  	\begin{array}{ll}
			    		  	U(t)=\disp\int_0^1k_x(1,y)u(y,t)dy+\gamma^\prime(1)X(t)+q(1,1)w(1,t)+\disp\int_0^1q_x(1,y)w(y,t)dy.
			    		  	\end{array}
			    		  	\end{equation}
			    		  		Here we use the fact that $k(1,1)=0$.
			    		  		
			    		  	Under \dref{conc}, the closed-loop system of \dref{cn} with boundary actuation \dref{lil} becomes
			    		  	\dref{o} subject to
			    		  	\begin{equation}\label{ne}
			    		  	z_x(1,t)=0.
			    		  	\end{equation}
			    		  	\section{Well-posedness of the closed-loop systems}
			    		  	We start with the Dirichlet actuation. Consider the closed-loop target system \dref{o} and \dref{d}.  We introduce the Hilbert space
			    		  		  	$H_R^1(0,1)=\{f\in H^1(0,1)|f(1)=0\}$ and $H=L^2(0,1)$.
			    		  		  		Let  $\mathcal H=\mathbb{R}^n\times L^2(0,1)$ be the Hilbert space with the norm:
			    		  		  		$$\|(f,g)\|^2_{\mathcal H}=|f|^2_{\mathbb{R}^n}+\|g\|^2_{L^2(0,1)}.$$
			    		  		  			
			    		  		  	While being viewed
			    		  		  	over the time segment $[0,h_0]$,  the system
			    		  		  	can be rewritten as the differential equation:
			    		  		  	\begin{equation}\label{qq}
			    		  		  		\left\{\begin{array}{ll}
			    		  		  			\dfrac{d}{dt}Y(\cdot,t)=\mathcal{A}_zY(\cdot,t)+\mathcal{A}_1Y(\cdot,t-\tau(t)),\\
			    		  		  		Y(\cdot,\theta)=(f(\theta),\varphi(\cdot,\theta)), \theta\in [-h,0]
			    		  		  		\end{array}\right.
			    		  		  	\end{equation}
			    		  		  	in $\mathcal H$, where  the system operator $\mathcal A_z:D(\mathcal{A}_z)\subset \mathcal H\rightarrow \mathcal H$ is defined by
			    		  		  	\begin{equation}
			    		  		  	\left\{\begin{array}{ll}
			    		  		  	 \mathcal{A}_z(X,z)=&[(A+BK)X+Bz(0),z_{xx}-cz],\\
			    		  		  	 	D(\mathcal{A}_z)=&\{(X,z)\in \mathbb{R}^n\times (H^2(0,1)\cap H_R^1(0,1))|\\&z_x(0)=0\},
			    		  		  	 \end{array}\right.
			    		  		  \end{equation}
			    		  		  and the bounded operator $\mathcal A_1:\mathcal{H}\rightarrow \mathcal{H}$ is defined by 
			    		  		  	$$\begin{array}{ll}
			    		  		  	\mathcal A_1(X,z)=[A_1X,a_2z-g(\cdot)(A_1-a_2I)X],
			    		  		  	\end{array}$$
			    		  		  	where $g(x)=\gamma(x)-\disp\int_0^xq(x,y)\gamma(y)dy$.
			    		  		  	
			    		  		  	A straightforward computation gives  
			    		  		  	 	\begin{equation}
			    		  		  	 	\left\{\begin{array}{ll}
			    		  		  	 	\mathcal{A}_z^*(X^*,z^*)&=[(A+BK)^\top X,z^{*}_{xx}-cz],\\
			    		  		  	 	D(\mathcal{A}_z^*)=&\{(X^*,z^*)\in \mathbb{R}^n\times (H^2(0,1)\cap H_R^1(0,1))|\\&z^*_x(0)=-B^\top X^*\},
			    		  		  	 	\end{array}\right.
			    		  		  	 	\end{equation}
			    		  		  	 where $\mathcal A_z^*$ is	the adjoint operator of $\mathcal A_z$. 
			    		  		  	 
			    		  		  		By the arguments of \cite{JMwang}, it can be shown that there is a sequence of eigenfunctions of $\mathcal A_z^*$ which forms a Riesz basis for $\mathcal H$  and hence $\mathcal {A}_z^* $
			    		  		  		generates an exponentially stable semigroup. Then by Proposition 2.8.1 and Proposition 2.8.5 of \cite{George},
			    		  		  		we obtain that $\mathcal {A}_z$ generates a $C_0$-semigroup.
			    		  		  	
			    		  		  		Define the initial  conditions in the space
			    		  		  		$$W\triangleq C([-h,0], D(\mathcal{A}_z))\cap C^1([-h,0],\mathcal H).$$
			    		  		  	The inhomogeneous term $\mathcal A_1Y(\cdot,t-\tau(t))$ is of
			    		  		  	class $C^1$ on $[0,h_0]$. By Theorem 3.1.3 of \cite{Curtain},  for any initial value $(X(\theta),z(\cdot,\theta)) \in  W$, the closed-loop target system admits a unique classical solution
			    		  		  	$(X(t), z(\cdot,t))$ for all $t\in [0,h_0]$.\par
			    		  		  	
			    		  		  		The same line of reasoning is step-by-step applied to the time
			    		  		  		segments $[h_0,2h_0]$, $[2h_0,3h_0]$, $[3h_0,4h_0]$, $\cdots$. Following this procedure, we obtain that there exists a unique classical solution $(X(t), z(\cdot,t))$ for all $t\geq 0$ with the initial condition $(X(\theta),z(\cdot,\theta)) \in  W$ (see e.g. \cite{Emilia}).
			    		  		  		
			    		  		  	Consider next the closed-loop target system \dref{o}, \dref{ne} under the Neumann actuation. Let $\mathcal H_1=\mathbb{R}^n\times H^1(0,1)$ be the Hilbert space with the norm:
			    		  		  	$$\|(f,g)\|^2_{\mathcal H_1}=|f|^2_{\mathbb{R}^n}+\|g\|^2_{L^2(0,1)}+\|g^\prime\|^2_{L^2(0,1)}.$$
			    		  		   The existence and uniqueness of the solution of the system \dref{o} subject to \dref{ne} can be easily obtained by applying the same procedure. But the expression of the domain $D(\mathcal{A}_z)$ should be changed into
			    		  		  	$$
			    		  		  	\begin{array}{ll}
			    		  		  	D(\mathcal{A}_z)=&\{(X,z)\in \mathbb{R}^n\times H^2(0,1)|z^\prime(0)=
			    		  		  	z^\prime(1)=0\}
			    		  		  	\end{array}
			    		  		  	$$
			    		  		  	and
			    		  		  	$$
			    		  		  	\begin{array}{ll}
			    		  		  	W= C([-h,0], D(\mathcal{A}_z))\cap C^1([-h,0],\mathcal H_1).
			    		  		  	\end{array}
			    		  		  	$$	

			    		  	 \begin{remark}
 By using the transformation \dref{wei} and \dref{lyp}, we establish the well-posedness of the closed-loop original system \dref{an} under the Dirichlet or Neumann actuation.
			    		  	 	
			    		  	 	For the case of Dirichlet actuation, we define
			    		  	 	$$
			    		  	 	\begin{array}{ll}
			    		  	 D(\mathcal{A}_u)=\left\{(X,u)\in \mathbb{R}^n\times H^2(0,1)|u^\prime(0)=0,\right.\\
			    		  	\left. u(1)=\int_0^1k(1,y)u(y)dy+\gamma(1)X+\int_0^1q(1,y)[u(y)-\int_0^yk(y,s)u(s)ds-\gamma(y)X]dy
			    		  	 \right\},\\
			    		  	 	W_1\triangleq C([-h,0], D(\mathcal{A}_u))\cap C^1([-h,0],\mathcal H).
			    		  	 	\end{array}
			    		  	 	$$
			    		  	 	Thus, for any initial value $(X(\theta),u(\cdot,\theta)) \in  W_1$, the closed-loop original system admits a unique classical solution
			    		  	 	$(X(t), u(\cdot,t))$ for all $t\geq 0$.
			    		  	 	
			    		  	 		For the case of Neumann actuation, we define
			    		  	 		$$
			    		  	 		\begin{array}{ll}
			    		  	 		D(\mathcal{A}_u)=\left\{(X,u)\in \mathbb{R}^n\times H^2(0,1)|u^\prime(0)=0,\right.\\
			    		  	 		\left.u^\prime(1)=\int_0^1k_x(1,y)u(y)dy+\gamma^\prime(1)X+q(1,1)[u(1)
			    		  	 		-\int_0^1k(1,y)u(y)dy-\gamma(1)X]\right.\\\left.+\int_0^1q_x(1,y)[u(y)
			    		  	 		-\int_0^yk(y,s)u(s)ds-\gamma(y)X]dy\right\},\\
			    		  	 		W_1= C([-h,0], D(\mathcal{A}_u)))\cap C^1([-h,0],\mathcal H_1).
			    		  	 		\end{array}
			    		  	 		$$
			    		  	 		Thus,  well-posedness of the closed-loop original system can be obtained.
			    		  	 	\end{remark}
			    		  	 	
\section{Stability analysis}
			    		  	 	In Theorem 2 of \cite{Tomoaki}, a delay-independent condition for the exponential stability of target system, which is described by reaction diffusion equation with state delay, has been shown by applying Lyapunov-Razumikhin theory. In this section,  we will derive an exponential bound on the solution of the target coupled system via Halanay's inequality. This solution bound will allow to find a domain of attraction in the case of actuator saturation.
			    		  	 	\subsection{Stability of system \dref{o} subject to \dref{d}}
			    		  	 	For the case of Dirichlet actuation, we choose the  Lyapunov functions of the form
			    		  	 	\begin{equation}\label{xin}
			    		  	 	 V(t)=X^\top PX+{p_1}\int_0^1z^2(x,t)dx,
			    		  	 \end{equation}
			    		  	 	where the $n\times n$ matrix $P=P^\top>0$, and the parameter $p_1>0$ will be chosen later.
			    		  	 	We aim to derive conditions that satisfy the Halanay inequality.
			    		  	 	\begin{lemma}(Halanay's Inequality \cite{Halanay})\label{H}
			    		  	 	Let $0<\delta_1< \delta_0$ and let $V:[-h,\infty)\to [0,\infty)$ be an absolutely continuous function that satisfies
			    		  	 	\begin{equation}
			    		  	 	\dot{V}(t)\le -2\delta_0 V(t)+2\delta_1\sup\limits_{-h\le\theta\le 0}{V}(t+\theta),\; t\geq 0.
			    		  	 	\end{equation}
			    		  	 	 Then
			    		  	 	\begin{equation}\label{xing}
			    		  	 	V(t)\le e^{-2\delta t}\sup\limits_{-h\le\theta\le 0}{V}(\theta),
			    		  	 	\end{equation}
			    		  	 	where $\delta$ is a unique solution of
			    		  	  	$\delta=\delta_0-\delta_1e^{2\delta h}$.
			    		  	 	\end{lemma}
			    		  	 	We will employ further Wirtinger's Inequality:
			    		  	 	\begin{lemma}(Wirtinger's Inequality) Let $z\in H^1(0,L)$ be a scalar function with $z(0)=0$
			    		  	 		or $z(L)=0$. Then
			    		  	 		\begin{equation}
			    		  	 	\int_0^L z^2(\xi)d\xi\le \dfrac{4L^2}{\pi^2}\int_0^L\left[\dfrac{dz}{d\xi}\right]^2d\xi.
			    		  	 		\end{equation}
			    		  	 	\end{lemma}
			    		  		\begin{proposition}\label{ni}
			    		  		{\rm (i)}	Given gains $K$ and $c$, and tuning parameters  $r>0$, $0<\delta_1<\delta_0$, let there exist scalars $p_1>0$, $0<\lambda\le 2p_1$ and an $n\times n$ matrix  $P>0$ that satisfy the following linear matrix inequalities:
			    		  			\begin{equation}\label{zz}
			    		  			\Theta_1\triangleq \Xi+p_1r^{-1}R< 0,
			    		  			\end{equation}
			    		  			\begin{equation}\label{zzz}
			    		  			\Theta_2\triangleq\left[\begin{array}{cc}(-2c+2\delta_0+r-\dfrac{\pi^2}{2})p_1+\dfrac{\pi^2}{4}\lambda&a_2p_1\\
			    		  			\ast&-2\delta_1p_1
			    		  			\end{array}
			    		  			\right]< 0,	
			    		  			\end{equation}
			    		  			where
			    		  			\begin{equation}\label{e}
			    		  		   \Xi=\left[\begin{array}{ccc}\theta_{11}&PA_1&PB\\\ast&-2\delta_1P &0\\
			    		  			\ast&\ast&-\lambda
			    		  			\end{array}
			    		  			\right],
			    		  			\end{equation}
			    		  			\begin{equation}\label{p}
			    		  			R={\rm diag}\{0,\zeta(A_1-a_2I)^\top(A_1-a_2I),0\},
			    		  			\end{equation}
			    		  		\begin{equation}
			    		  			\theta_{11}=P(A+BK)+(A+BK)^\top P+2\delta_0P,
			    		  		\end{equation}
			    		  		\begin{equation}\label{qian}
			    		  		   \zeta\triangleq(1+\max\limits_{0\le y\le  x\le 1}|q(x,y)|)^2(\max\limits_{0\le x\le 1}|\gamma(x)|)^2.
			    		  		\end{equation}
			    		  			 Then, 
			    		  			  for all $h_0>0$, $h>0$ and  $\tau(t)\in [h_0, h]$, the system \dref{o} subject to \dref{d} with initial conditions $(f,\varphi)\in W$ is exponentially stable with a decay rate $\delta$ in the sense that
			    		  			\dref{xing} holds,
			    		  			where $\delta$ is a unique solution of
			    		  			$\delta=\delta_0-\delta_1e^{2\delta h}$. Moreover, if the strict LMIs \dref{zz} and \dref{zzz} with $\delta_0=\delta_1>0$ hold, then for all $h_0>0$, $h>0$ and  $\tau(t)\in [h_0,h]$, the system \dref{o} subject to \dref{d} is exponentially stable with a small enough decay rate.\\
			    		  			{\rm (ii)}Assume now that $A_1$ is a scalar matrix, i.e. $A_1=a_1I$, where $a_1$ is some constant. Then given any $\delta>0$, the exponential stability of the system 
			    		  			\dref{o} subject to \dref{d} with the decay rate $\delta>0$  can be achieved by appropriate choice of design parameters $c$ and $K$.
			    		  		\end{proposition}
			    		  		\begin{proof}
			    		  			{\rm (i)}	Differentiating $V$ along \dref{o} and \dref{d} we find
			    		  			$$
			    		  			\begin{array}{ll}
			    		  			\dot V(t)	\disp=2p_1\int_0^1z(x,t)z_t(x,t)dx+X^\top(t)P\dot{X}(t)+\dot{X}^\top(t)PX(t).
			    		  			\end{array}
			    		  			$$
			    		  			Integration by parts and substitution of the boundary conditions \dref{o} and \dref{d} lead to
			    		  			\begin{equation}\label{w}
			    		  			\begin{array}{ll}
			    		  			&\dot{V}(t)+2\delta_0V(t)-2\delta_1\sup\limits_{-h\le\theta\le 0}{V}(t+\theta)\\&\le \disp-2p_1\int_0^1z_x^2(x,t)dx+2a_2p_1\int_0^1z(x,t)z(x,t-\tau(t))dx\disp-2p_1c\int_0^1z^2(x,t)dx\\&-2\delta_1p_1\disp\int_0^1z^2(x,t-\tau(t))dx-2\delta_1X^\top(t-\tau(t))PX(t-\tau(t))\\&-2p_1\disp\int_0^1z(x,t)[\gamma(x)-\disp\int_0^xq(x,y)\gamma(y)dy]dx(A_1-a_2I)X(t-\tau(t))\\&+X^\top(t)[P(A+BK)+(A+BK)^\top P]X(t)+2X^\top(t) PBz(0,t)+2X^\top(t) PA_1X(t-\tau(t))\\&
			    		  			\disp+2\delta_0p_1\int_0^1z^2(x,t)dx+2\delta_0X^\top(t)PX(t).
			    		  			\end{array}
			    		  			\end{equation}
			    		  			From Sobolev's inequality and Wirtinger's inequality, we have
			    		  			\begin{equation}\label{t}
			    		  			-\int_0^1z_x^2(x,t)dx	\le	-z^2(0,t),
			    		  			\end{equation}
			    		  			\begin{equation}\label{xi}
			    		  			-\int_0^1z_x^2(x,t)dx	\le -\dfrac{\pi^2}{4}\int_0^1z^2(x,t)dx.
			    		  			\end{equation}
			    		  			 Multiplying the inequality \dref{t} by a constant $\lambda\in (0,2p_1]$ and multiplying the inequality \dref{xi} by $2p_1-\lambda$ on both sides and summing, we obtain that
			    		  			 \begin{equation}\label{xian}
			    		  			 -2p_1\int_0^1z_x^2(x,t)dx\le -\lambda z^2(0,t)-\dfrac{\pi^2}{4}(2p_1-\lambda)\int_0^1z^2(x,t)dx.
			    		  			 \end{equation}
			    		  			 As $\gamma(x)$, $q(x,y)$ are continuous functions bounded on any compact,
			    		  			 the following inequality can be obtained:
			    		  			 \begin{equation}\label{new1}
			    		  			 \begin{array}{ll}\disp\int_0^1\left|\gamma(x)-\disp\int_0^xq(x,y)\gamma(y)dy\right|^2dx\le (1+\max\limits_{0\le y\le  x\le 1}|q(x,y)|)^2(\max\limits_{0\le x\le 1}|\gamma(x)|)^2=\zeta,
			    		  			 \end{array}
			    		  			 \end{equation}
			    		  			 which together with Young's inequality implies 
			    		  			 \begin{equation}\label{new}
			    		  			 \begin{array}{ll}
			    		  			 &-2p_1\disp\int_0^1z(x,t)[\gamma(x)-\disp\int_0^xq(x,y)\gamma(y)dy]dx(A_1-a_2I)X(t-\tau(t))\\
			    		  			 &\le p_1\left[r\disp\int_0^1z^2(x,t)dx+r^{-1}X^\top(t-\tau(t))SX(t-\tau(t))\right],
			    		  			 \end{array}
			    		  			 \end{equation} 
			    		  			 where 
			    		  			 \begin{equation}\label{new3}
			    		  			 S=\zeta(A_1-a_2I)^\top (A_1-a_2I).
			    		  			 \end{equation}
			    		  			Set $\eta_1(t)={\rm col}\{X(t),X(t-\tau(t)),z(0,t)\}$, $\eta_2(t)={\rm col}\{z(x,t),z(x,t-\tau(t))\}$. Then substituting \dref{xian}, \dref{new}  into \dref{w} yields
			    		  			$$
			    		  			\begin{array}{ll}
			    		  			\dot{V}(t)+2\delta_0V(t)-2\delta_1\sup\limits_{-h\le\theta\le 0}{V}(t+\theta)\le \sum\limits_{i=1}^{2}\int_0^1\eta_i^\top(t) \Theta_i\eta_i(t) dx \le 0
			    		  			\end{array}
			    		  			$$
			    		  			if the LMIs $\Theta_1< 0$ and $\Theta_2< 0$ hold. Therefore, the feasibility of $\Theta_1< 0$ and $\Theta_2< 0$ guarantees that  the Halanay inequality \dref{xing} holds meaning that the system \dref{o} subject to \dref{d} is exponentially stable. \\
			    		  			The feasibility of strict inequalities \dref{zz} and \dref{zzz} with $\delta_1=\delta_0>0$ implies feasibility of these inequalities with $\bar \delta_0$ and $\bar \delta_1$ given by                      
			    		  			$\bar\delta_0=\delta_0+\epsilon>\delta_0=\bar\delta_1$ for small enough $\epsilon>0$.
			    		  		Since Halanay's inequality holds with $\bar\delta_0$ and $\bar \delta_1$, the system is exponentially stable with a small enough decay rate.
			    		  	
			    		 	{\rm (ii)} The decay rate bound can be enlarged if for given $\delta_1>0$ we can enlarge $\delta_0>\delta_1$ subject to $\Theta_1< 0, \ \Theta_2< 0$.
			    		  	Applying Schur complement theorem, we obtain 
			    		  	
			    		  	\begin{equation}\label{b}
			    		  	\begin{array}{ll}
			    		  	\Xi<0\Longleftrightarrow 
			    		  	
			    		  	P(A+BK)+(A+BK)^\top P+\lambda^{-1}PBB^\top P+[2\delta_0+(2\delta_1)^{-1}a_1^2] P< 0.
			    		  	\end{array}
			    		  	\end{equation}
			    		  	Multiplying the last inequality by $Q=P^{-1}$ from left and right we arrive at
			    		  	\begin{equation}\label{dD}
			    		  		\begin{array}{ll}
			    		  	\Xi<0\Longleftrightarrow  (A+BK) Q+Q(A+BK)^\top+\lambda^{-1}BB^\top+[2\delta_0+(2\delta_1)^{-1}a_1^2]Q<0.
			    		  	\end{array}
			    		  	\end{equation}
			    		  		Since $(A, B)$ is controllable, for any $0<\delta_1<\delta_0$ and $0<\lambda\le 2$, we can choose  $K$ such that 
			    		  		 Lyapunov inequality \dref{dD} has a solution $Q>0$.  Then there exist  large enough  $r>0$ and $p_1=1$ such that \dref{zz} holds. 
			 
			    		  	By Schur complement theorem,
			    		  	\begin{equation}\label{wo}
			    		  		\begin{array}{ll}
			    		  			\Theta_2< 0 \Longleftrightarrow -2c+2\delta_0-\dfrac{\pi^2}{4}(2-\lambda p_1^{-1})+r+ (2\delta_1)^{-1}a_2^2< 0.
			    		  		\end{array}
			    		  	\end{equation}
			    		  	With the chosen above parameters $\delta_0$, $\delta_1$, $p_1$, $\lambda$ and $r$, \dref{wo} always holds for large enough $c$.  
			    		   Thus, given $h$,  any decay rate bound may be achieved by appropriate choice of design parameters $c$ and $K$.  
			    		  		\end{proof}
			    		 	\begin{remark}\label{chen}
			    		  		Less conservative delay-dependent stability conditions for system \dref{o} subject to \dref{d} with fast varying delays can be derived by using Lypunov-Krasovskii approach similar to \cite{Emilia4,Emilia2}. In fact, one can consider the following Lypunov-Krasovskii functional
			    		  	$$
			    		  		\begin{array}{ll}
			    		  	&	V(t)=X^\top(t) PX(t)+\disp\int_{t-h}^t e^{-2\delta_0 (t-s)}X^\top(s) SX(s)ds\\&+h\disp\int_{-h}^{0}\int_{t+\theta}^te^{-2\delta_0 (t-s)}\dot X^\top (s)R\dot X(s)dsd\theta +{p_1}\int_0^1z^2(x,t)dx
			    		  		\end{array}
			    		  		$$ 
			    		  		combined with the Halanay inequality, where $P,S,R>0$ are some matrices, and $p_1>0$ is a constant. The resulting conditions will be always feasible for small enough $h$
			    		  		provided $(A+A_1,B)$ is controllable.
			    		  	\end{remark}
			    		  		\begin{remark}The original system \dref{an}  is equivalent to system \dref{cn} under the invertible transformation \dref{wei}, and \dref{lyp}. Therefore, 
			    		  			under the conditions of Proposition \ref{ni}, for the original system \dref{an}, the same decay rate can be guaranteed  by the controller $U(t)$ given by \dref{con} .
			    		  		\end{remark}
			    		  		\subsection{Stability of system \dref{o} subject to \dref{ne}}
			    		  		For the case of Neumann actuation, we choose the Lyapunov function
			    		  		$$
			    		  		 V_1(t)= V(t)+{p_2}\|z_x\|^2=X^\top PX+{p_1}\|z\|^2+{p_2}\|z_x\|^2.
			    		  	$$
			    		  		where the $n\times n$ matrix $P=P^\top>0$,  the parameters $p_1>0$ and $p_2>0$
			    		  		will be chosen later, and $V(t)$ is defined by \dref{xin}.
			    		  		\begin{remark}
			    		  		Similar to the case of Dirichlet actuation, for the proof of the stability of system \dref{o} subject to \dref{ne}, we can choose $p_2=0$. For finding a domain of attraction under Neumann actuation in the presence of
			    		  			actuator saturation, we need $p_2>0$ (see Section 5).
			    		  		\end{remark}
			    		  		\begin{proposition}\label{nii}
			    		  			{\rm (i)}Given gains $K$ and $c$,   and tuning parameters $r>0$, $0<r_1<2$, $0<\delta_1< \delta_0$, let there exist an $n\times n$ matrix $P>0$, and scalars $p_1>0$, $p_2>0$, $\lambda>0$ and $\lambda_1\geq0$ that satisfy the LMIs 
			    		  				\begin{equation}\label{yyyy}
			    		  				\begin{array}{ll}
			    		  				\bar\Theta_1\triangleq \Theta_1+p_2r_1^{-1}R=\Xi+(p_1r^{-1}+p_2r_1^{-1})R< 0,
			    		  				\end{array}
			    		  				\end{equation}
			    		  				\begin{equation}\label{ppp}
			    		  				\bar\Theta_2\triangleq\left[\begin{array}{ccc}(-2c+2\delta_0+r)p_1+2\lambda&{a_2p_1}&0\\
			    		  				\ast&-2\delta_1p_1&-{a_2p_2}\\
			    		  				\ast&\ast&\theta_{33}
			    		  				\end{array}
			    		  				\right]< 0,
			    		  				\end{equation}
			    		  				and the inequality
			    		  				\begin{equation}\label{ke}
			    		  				-2p_1-2p_2c+\lambda+2\delta_0p_2-\dfrac{\pi^2}{4}\lambda_1\le 0,
			    		  				\end{equation}
			    		  				where $\Xi$, $R$ are defined by \dref{e} and \dref{p} respectively,
			    		  				$$
			    		  				\theta_{33}=-(2-r_1)p_2+\lambda_1.
			    		  				$$
			    		  			 Then, for all  $h_0>0$, $h>0$ and  $ \tau(t) \in [h_0,h]$, the system \dref{o} subject to \dref{ne} with initial condition
			    		  			$(f,\varphi)\in W$  is exponentially stable with a decay rate $\delta$,
			    		  			where $\delta$ is a unique solution of
			    		  			$\delta=\delta_0-\delta_1e^{2\delta h}$.
			    		  			Moreover, if \dref{yyyy}, \dref{ppp} and \dref{ke} hold with $\delta_0=\delta_1>0$, then for all  $h_0>0$ and $h>0$, the system \dref{o} subject to \dref{ne} is exponentially stable with a small enough decay rate for all $\tau(t)\in [h_0,h]$.\\
			    		  			{\rm (ii)}Assume now that $A_1$ is a scalar matrix, i.e. $A_1=a_1I$, where $a_1$ is some constant. Then given any $\delta>0$, the exponential stability of the system 
			    		  			\dref{o} subject to \dref{ne} with the decay rate $\delta>0$  can be achieved.
			    		  		\end{proposition}
			    		  		\begin{proof}	
			    		  	(\rm i)	Taking the time derivative of the Lyapunov function along the solution of \dref{o} subject to \dref{ne}, and from \dref{w} we get
                            	\begin{equation}\label{www}
			    		  		\begin{array}{ll}
			    		  		&\dot{V}_1(t)+2\delta_0V_1(t)-2\delta_1\sup\limits_{-h\le\theta\le 0}{V}_1(t+\theta)\\&\le \disp-2p_1\int_0^1z_x^2(x,t)dx+2a_2p_1\int_0^1z(x,t)z(x,t-\tau(t))dx
			    		  		-2p_1c\int_0^1z^2(x,t)dx\\&\disp-2p_2c\int_0^1z_x^2(x,t)dx
			    		  		\disp -2p_2\int_0^1z_{xx}^2(x,t)dx-2a_2p_2\int_0^1z_{xx}(x,t)z(x,t-\tau(t))dx\\
			    		  		&-2\disp\int_0^1[p_1z(x,t)-p_2z_{xx}(x,t)][\gamma(x)-\disp\int_0^xq(x,y)\gamma(y)dy]dx(A_1-a_2I)X(t-\tau(t))\\
			    		  		&+X^\top(t)[P(A+BK)+(A+BK)^\top P]X(t)\disp+2X^\top(t) PBz(0,t)+2X^\top(t) PA_1X(t-\tau(t))\\
			    		  		&\disp+2\delta_0p_1\int_0^1z^2(x,t)dx+2\delta_0p_2\int_0^1z_x^2(x,t)dx\disp+2\delta_0X^\top(t)PX(t)-2\delta_1X^\top(t-\tau(t))PX(t-\tau(t))\\&\disp-2\delta_1p_2\int_0^1z_x^2(x,t-\tau(t))dx-2\delta_1p_1\int_0^1z^2(x,t-\tau(t))dx.
			    		  		\end{array}
			    		  		\end{equation}
			    		  		From Young's inequality, we have \dref{new} and
			    		  		\begin{equation}
			    		  		\begin{array}{ll}
			    		  		&\disp2\int_0^1p_2z_{xx}(x,t)[\gamma(x)-\int_0^xq(x,y)\gamma(y)dy]dx(A_1-a_2I)
			    		  	X(t-\tau(t))\\
			    		  		&\disp\le p_2[r_1\int_0^1z_{xx}^2(x,t)dx+r_1^{-1}X^\top(t-\tau(t))SX(t-\tau(t))],
			    		  		\end{array}
			    		  		\end{equation}
			    		  		where  $r_1>0$ and $S$ is defined by \dref{new3}.\\
			    		  		By using Agmon's and Wirtinger's inequalities, we have
			    		  		$$
			    		  		|z(0,t)|^2\le 2\|z\|^2+\|z_x\|^2,	
			    		  		$$
			    		  		$$
			    		  		\|z_x\|^2\le \dfrac{4}{\pi^2}\|z_{xx}\|^2.
			    		  		$$
			    		  		Hence,
			    		  		\begin{equation}\label{lin}
			    		  		0\le\lambda[ \|z_x\|^2+2\|z\|^2-|z(0,t)|^2],
			    		  		\end{equation}
			    		  		\begin{equation}\label{zhi}
			    		  		0\le \lambda_1[\|z_{xx}\|^2-\dfrac{\pi^2}{4}\|z_x\|^2],
			    		  		\end{equation}
			    		  	where $\lambda,\lambda_1>0$ are some constants.
			    		  	
			    		  	 We add \dref{lin} and \dref{zhi} to \dref{www}.
			    		  		Set $\eta_1(t)={\rm col}\{X(t),X(t-\tau(t)),z(0,t)\}$,  $\eta_2(t)={\rm col}\{z(x,t),\\z(x,t-\tau(t)),z_{xx}(x,t)\}$. Let $\Theta_1$ be defined by \dref{zz} and $R$ by \dref{p}. Then we obtain
			    		  		\begin{equation}
			    		  		\begin{array}{ll}
			    		  		&\dot{V}_1(t)+2\delta_0V_1(t)-2\delta_1\sup\limits_{-h\le\theta\le 0}{V}_1(t+\theta)\\&\le \sum\limits_{i=1}^2\disp\int_0^1\eta_i^\top(t) \bar\Theta_i\eta_i(t)dx-(2p_1+2p_2c-\lambda-2\delta_0p_2+\dfrac{\pi^2}{4}\lambda_1)\disp\int_0^1z_x^2(x,t)dx\le 0
			    		  		\end{array}
			    		  		\end{equation}
			    		  		if the LMIs $\bar\Theta_1< 0$, $\bar\Theta_2< 0$ are feasible and the inequality \dref{ke} holds. Application of Halanay's inequality, completes the proof of ({\rm i}).

			    		  		(\rm ii)  By (ii) of Proposition \ref{ni}, $\Theta_1<0$ is feasible for given $0<\delta_1<\delta_0$ and appropriate $K$.  Then  for $r_1=1$ and small enough $p_2>0$,  $\bar \Theta_1<0$ is feasible. \\
			    		  	Now given $0<\delta_1<\delta_0$,  $\lambda>0$,  $p_1=1$, $p_2>0$, and $\lambda_1\geq 0$ such that $\theta_{33}<0$, we show that \dref{ppp}  and \dref{ke} are feasible for appropriate choice of large enough $c>0$. 
			    		  	For \dref{ke}, this is evident. For \dref{ppp}, this is true by Schur complements theorem. 
			    		  		 \end{proof}

			    		  			\begin{remark}
			    		  				 For simplicity only, in the cascade model
			    		  				we consider a constant coefficient $a$ of the undelayed term $au(x,t)$.
			    		  				For the variable $a(s)$, one have to modify kernels of the transformations similarly to \cite{Tomoaki}. Halanay's inequality is applicable for the resulting target system.
			    		  			\end{remark}
			    		  			\section{Control under saturation: regional stabilization}
			    		  				In this section, we consider the system \dref{an} with the control law which is subject to the following amplitude constraint:
			    		  				\begin{equation}
			    		  				|U(t)|\le \bar{u}.
			    		  				\end{equation}	
			    		  				Denoting the state trajectory of \dref{an} subject to Dirichlet or Neumann boundary actuation with the initial condition  $(X_0,u_0)\triangleq (f(\theta),\psi(\cdot,\theta))\in W_1$ by $(X(t;X_0),u(x,t;u_0))$.
			    		  				
			    		  					For the case of Dirichlet actuation, the domain of attraction  of the closed-loop original system is then the set
			    		  					\begin{equation}
			    		  					\label{we}
			    		  					\tilde{\mathcal{S}}=\left\{(X_0,u_0)\in W_1:\lim_{t\to \infty}\|(X(t;X_0),u(x,t;u_0))\|_{\mathcal H}=0\right\}.
			    		  				    \end{equation}	
			    		  					For the case of Neumann actuation, the domain of attraction  of the closed-loop original system is given by \dref{we}, where $\mathcal H$ is replaced by $\mathcal H_1$.
			    		  			\subsection{Dirichlet control under saturation}
			    		  			 We first find domain of attraction for the closed-loop target system. Denoting the state trajectory of closed-loop target system with the initial condition  $(X_0,z_0)\triangleq (f(\theta),\varphi(\cdot,\theta))\in W$ by $(X(t;X_0),z(x,t;z_0))$, the domain of attraction  of the closed-loop target system is then the set
			    		  				$$
			    		  				\mathcal{S}=\left\{(X_0,z_0)\in W:\lim_{t\to \infty}\|(X(t;X_0),z(x,t;z_0))\|_{\mathcal H}=0\right\}.
			    		  				$$	
			    		  			 We will obtain an estimate $\mathcal{X}_\beta\subset \mathcal{S}$ on the domain of attraction, where
			    		  			 $$
			    		  			 \mathcal{X}_\beta=\left\{(X_0,z_0)\in W:\max_{[-h,0]}|X_0|^2+\max_{[-h,0]}\|z_0\|^2\le \beta^{-1}\right\},
			    		  			$$
			    		  			 $\beta>0$ is a scalar that will be minimized in the sequel.
			    		  			
			    		  			 We design the state feedback controller in the following form:
			    		  			 \begin{equation}\label{s}
			    		  			 U_{sat}(t)={\rm sat}(U(t),\bar{u}),
			    		  			 \end{equation}
			    		  			 where $U(t)$ is given by \dref{con}.\\
			    		  			 Applying the latter control law \dref{s}, we represent the saturated closed-loop target system as the system \dref{o} with the following boundary condition:
			    		  	          \begin{equation}\label{sat}
			    		  			  z(1,t)={\rm sat}(U(t),\bar{u})-U(t).
			    		  			  \end{equation}	
			    		  		From \dref{con}, $U(t)$ admits the following representation:
			    		  		$$
			    		  		\begin{array}{ll}
			    		  		U(t)
			    		  		&\disp=\int_0^1n(1,y)w(y,t)dy+\psi(1)X(t)+\int_0^1l(1,y)z(y,t)dy\\
			    		  		&\disp=\int_0^1n(1,y)\left[z(y,t)+\int_0^yl(y,s)z(s,t)ds\right]dy
			    		  		+\psi(1)X(t)+\int_0^1l(1,y)z(y,t)dy,
			    		  		\end{array}
			    		  	$$
			    		  		provided saturation is avoided.\\
			    		  			Denote
			    		  			$$
			    		  			\begin{array}{ll}
			    		  			c_1=|\psi(1)|,
			    		  			c_2=\max\limits_{0\le y\le 1}|n(1,y)|\left(1+\max\limits_{0\le y\le x\le 1}|l(x,y)|\right)+\max\limits_{0\le y\le 1}|l(1,y)|.
			    		  			\end{array}	
			    		  			$$
			    		  		Due to \dref{xia} and \dref{n}, $n(x,y)$ and $l(x,y)$ are continuous functions bounded on any compact. Then Jensen's inequality implies
			    		  		$$
			    		  		\begin{array}{ll}
			    		  		|U(t)|&\le c_1|X|+c_2\|z\|.
			    		  		\end{array}
			    		  		$$
			    		  		Applying Young's inequality,
			    		  		we obtain
			    		  		\begin{equation}\label{2}
			    		  		\begin{array}{ll}
			    		  		|U(t)|^2&\le 2c_1^2|X|^2+2c_2^2\|z\|^2.
			    		  		\end{array}
			    		  		\end{equation}
			    		  		
			    		  		Given $\bar u>0$, we define the following set:
			    		  		\begin{equation}\label{1}
			    		  		\mathcal{L}(c_1,c_2,\bar u)=\left\{(X,z)\in \mathcal H:c_1^2|X|^2+ c_2^2\|z\|^2\le \dfrac{\bar{u}^2}{2}\right\}.
			    		  		\end{equation}
			    		  		From the inequality \dref{2} and the definition \dref{1}, we can obtain the following implication:
			    		  		if $(X,z) \in \mathcal{L}(c_1,c_2,\bar u)$, then $|U(t)|\le \bar{u}$, and the saturation is avoided. Thus, the system  \dref{o} subject to \dref{sat} admits the linear representation \dref{o} subject to \dref{d}.
			    		  		
			    		  		From  Proposition \ref{ni}, we find that if there exist $0<\delta_1=\delta_0$ such that the strict LMIs \dref{zz}, \dref{zzz} are feasible, then the following inequality holds
			    		  		$$
			    		  		\begin{array}{ll}
			    		  		X^\top(t) PX(t)+{p_1}\disp\int_0^1z^2(x,t)dx	=V(t)\le \sup\limits_{-h\le\theta\le 0}{V}(\theta)\le
			    		  		\lambda_{\max}(P)\max\limits_{[-h,0]}|X_0|^2+p_1\max\limits_{[-h,0]}\|z_0\|^2,\\\;\hspace{12cm}\forall t\geq 0.
			    		        \end{array}
			    		  		$$	
			    		  		Hence, the inequalities: \begin{equation}\label{qi}
			    		  	P\le \beta I,\; p_1\le \beta
			    		  		\end{equation}
			    		  		guarantee that the trajectories $(X(t;X_0),z(x,t;z_0))$ starting from initial function $(X_0,z_0)\in \mathcal{X}_\beta$ remain within $\mathcal{X}_z$,
			    		  		where
			    		  		$$
			    		  		\mathcal{X}_z=\left\{(X,z)\in \mathcal H: X^\top(t) PX(t)+p_1\disp\int_0^1z^2(x,t)dx\le 1\right\}.
			    		  		$$
			    		  		The ``ellipsoid" $\mathcal{X}_z$ is contained in $\mathcal{L}(c_1,c_2,\bar u)$, if the following implication holds
			    		  		$$
			    		  		\begin{array}{ll}
			    		  		X^\top(t) PX(t)+p_1\disp\int_0^1z^2(x,t)dx\le1  \Longrightarrow c_1^2|X(t)|^2+c_2^2\|z(x,t)\|^2\le \dfrac{\bar{u}^2}{2}
			    		  		\end{array}
			    		  		$$
			    		  		for all $(X(t), z(x,t))$, i.e. if
			    		  		$$
			    		  		\begin{array}{ll}
			    		  		c_1^2|X(t)|^2+c_2^2\|z(x,t)\|^2\disp\le \dfrac{\bar{u}^2}{2}\left[X^\top(t) PX(t)+p_1\int_0^1z^2(x,t)dx\right].
			    		  		\end{array}
			    		  		$$
			    		  		The latter inequality is guaranteed if
			    		  		\begin{equation}\label{ww}
			    		  		P\dfrac{\bar u^2}{2}-c_1^2I\geq 0,\; p_1\dfrac{\bar u^2}{2}-c_2^2\geq 0.
			    		  		\end{equation}
			    		  		Therefore, the inequalities \dref{ww} guarantee the saturation avoidance, and together with Proposition \ref{ni} and condition \dref{qi} imply that
			    		  		$$
			    		  		\lim\limits_{t\to\infty}\|(X(t;X_0),z(x,t;z_0))\|_{\mathcal H}=0.
			    		  		$$
			    		  		
			    		  			Returning to the original system by the transformation \dref{wei} and \dref{lyp}, we have
			    		  			\begin{equation}\label{pq}
			    		  			\|z\|\le \left[1+\max\limits_{0\le y\le x\le 1}|q(x,y)|\right]\|w\|,
			    		  			\end{equation}
			    		  			\begin{equation}\label{pp}
			    		  			\|w\|\le \left[1+\max\limits_{0\le y\le x\le 1}|k(x,y)|\right]\|u\|+\left[\max\limits_{0\le x\le 1}|\gamma(x)|\right]|X|.
			    		  			\end{equation}
			    		  			Hence,
			    		  			\begin{equation}\label{yyl}
			    		  			|X|^2+\|z\|^2\le  M_1|X|^2+M_2\|u\|^2,
			    		  			\end{equation}
			    		  			where
			    		  			$$
			    		  			\begin{array}{ll}
			    		  			M_1=1+2\left[\max\limits_{0\le x\le 1}|\gamma(x)|\left(1+\max\limits_{0\le y\le x\le 1}|q(x,y)|\right)\right]^2,\\
			    		  			M_2=2\left[1+\max\limits_{0\le y\le x\le 1}|k(x,y)|\right]^2\left[1+\max\limits_{0\le y\le x\le 1}|q(x,y)|\right]^2.
			    		  			\end{array}
			    		  			$$
		    		  				Denote
			    		  			$$
			    		  				\begin{array}{ll}
			    		  				\mathcal{X}_{u}=\{(X_0,u_0)\in W_1:M_1\max\limits_{[-h,0]}|X_0|^2+M_2\max\limits_{[-h,0]}\|u_0\|^2
			    		  				\le \beta^{-1}\}.
			    		  				\end{array}
			    		  				$$
			    		  				It follows from the inequality \dref{yyl} that if the initial function of system \dref{an} with the Dirichlet boundary actuation \dref{s} satisfies $(X_0,u_0)\in \mathcal{X}_{u}$, then by backstepping transformation, the initial function of target system \dref{o} subject to \dref{sat}
			    		  				satisfies $(X_0,z_0)\in \mathcal{X}_\beta$. The following is thus obtained:
			    		  				\begin{theorem}\label{3.2}
			    		  				 	Given gains $K$ and $c$,  and tuning parameters $r>0$, $0<\delta_1=\delta_0$, let there exist an $n\times n$ matrix  $P>0$ and scalars $p_1>0$, $0\le\lambda\le 2p_1$  that satisfy the strict LMIs \dref{zz}, \dref{zzz}, \dref{qi} and \dref{ww}. Then for all $h_0>0$ and $h>0$, the classical solutions of \dref{an} with  Dirichlet boundary actuation \dref{s} starting from initial functions $(X_0,u_0)\in  \mathcal{X}_{u}$ converge to zero for all delays $\tau$ subject to \dref{ta}, i.e.
			    		  					$$
			    		  					\lim_{t\to \infty}\|(X(t;X_0),u(x,t;u_0))\|_{\mathcal H}=0.
			    		  					$$
			    		  				\end{theorem}	
			    		  				\subsection{Neumann control under saturation}
			    		  			 For the case of Neumann actuation, the domain of attraction  of the closed-loop target system is the set
			    		  				$$
			    		  					\mathcal{S}=\left\{(X_0,z_0)\in W:\lim_{t\to \infty}\|(X(t;X_0),z(x,t;z_0))\|_{\mathcal H_1}=0\right\}.
			    		  					$$	
			    		  				 We will obtain an estimate $\mathcal{X}_\beta\subset \mathcal{S}$ of the domain of attraction, where
			    		  				$$
			    		  				\begin{array}{ll}
			    		  				\mathcal{X}_\beta=\left\{(X_0,z_0)\in W:\max\limits_{[-h,0]}|X_0|^2+\max\limits_{[-h,0]}\|z_0\|^2+\max\limits_{[-h,0]}\|z_0^\prime\|^2\le \beta^{-1}\right\},
			    		  				\end{array}
			    		  				$$
			    		  				 $\beta>0$ is a scalar that will be minimized in the sequel.
			    		  				
			    		  				Then we design the state feedback controller in the following form
			    		  				\begin{equation}\label{as}
			    		  				U_{sat}(t)={\rm sat}(U(t),\bar{u}),
			    		  				\end{equation}
			    		  				where $U(t)$ is given by \dref{conc}.\\
			    		  				Applying the latter control law \dref{as}, we represent the saturated closed-loop target system into the system \dref{o} with the following boundary condition:
			    		  				\begin{equation}\label{lypin}
			    		  				z_x(1,t)={\rm sat}(U(t),\bar{u})-U(t).
			    		  				\end{equation}	
			    		  				In this case, from \dref{conc}, $U(t)$ admits the following representation:
			    		  				$$
			    		  				\begin{array}{ll}
			    		  				U(t)
			    		  				&\disp=\int_0^1n_x(1,y)w(y,t)dy+\psi^\prime(1)X(t)
			    		  				\disp+l(1,1)z(1,t)+\int_0^1l_x(1,y)z(y,t)dy\\
			    		  				&\disp=\int_0^1n_x(1,y)\left[z(y,t)+\int_0^yl(y,s)z(s,t)ds\right]dy+\psi^\prime(1)X(t)+l(1,1)z(1,t)\\
			    		  				&\disp+\int_0^1l_x(1,y)z(y,t)dy.
			    		  				\end{array}
			    		  				$$
			    		  				Here we use the fact that $n(1,1)=0$.
			    		  				
			    		  				Denote that
			    		  				$$
			    		  				\xi\triangleq\max\limits_{0\le y\le 1}|n_x(1,y)|(1+\max\limits_{0\le x\le y\le 1}|l(x,y)|)+\max\limits_{0\le y\le 1}|l_x(1,y)|.
			    		  			$$
			    		  				Applying Jensen's and Young's inequalities, we obtain
			    		  				$$
			    		  				\begin{array}{ll}
			    		  				|U(t)|&\le |l(1,1)||z(1,t)|+|\psi^\prime(1)||X(t)|+\xi\|z(x,t)\|.
			    		  				\end{array}
			    		  				$$
			    		  				By using Agmon's inequality, we have
			    		  				$$
			    		  				|z(1,t)|^2\le 2\|z(x,t)\|^2+\|z_x(x,t)\|^2.
			    		  				$$
			    		  				Denote that
			    		  				$$
			    		  				\begin{array}{ll}
			    		  				c_1=|\psi^\prime(1)|,\;c_2=\sqrt{2}|l(1,1)|+\xi,\;c_3=|l(1,1)|.
			    		  				\end{array}	
			    		  				$$
			    		  				Then,
			    		  				\begin{equation}\label{11}
			    		  				|U(t)|^2\le 3\left[c_1^2|X|^2+c_2^2\|z\|^2+c_3^2\|z_x\|^2\right].
			    		  				\end{equation}
			    		  				
			    		  				Given $\bar u>0$, we define the following set:
			    		  				\begin{equation}\label{10}
			    		  				\begin{array}{ll}
			    		  				\mathcal{L}(c_1,c_2,c_3,\bar u)=&\left\{(X,z)\in \mathcal H_1:c_1^2|X|^2+ c_2^2\|z\|^2\right.\\
			    		  			&\left.	+c_3^2\|z_x\|^2\le \dfrac{\bar{u}^2}{3}\right\}.
			    		  				\end{array}
			    		  				\end{equation}
			    		  				From the inequality \dref{11} and the definition \dref{10}, we can obtain:
			    		  				if $(X,z) \in \mathcal{L}(c_1,c_2,c_3,\bar u)$, then $|U(t)|\le \bar{u}$, and the saturation is avoided. Thus, the system \dref{o} subject to \dref{lypin} admits the linear representation \dref{o} subject to \dref{ne}.
			    		  				
			    		  				From Proposition \ref{nii}, we find that if there exist $0<\delta_1=\delta_0$ such that the LMIs \dref{yyyy}, \dref{ppp} and \dref{ke} are feasible, then the following inequality holds:
			    		  				$V(t)\le\sup\limits_{-h\le\theta\le 0}{V}(\theta)$, i.e. for all $ t\geq 0$,
			    		  				$$
			    		  				\begin{array}{ll}
			    		  					X^\top PX+{p_1}\|z\|^2+{p_2}\|z_x\|^2\le
			    		  					\lambda_{\max}(P)\max\limits_{[-h,0]}|X_0|^2+{p_1}\max\limits_{[-h,0]}\|z_0\|^2+{p_2}\max\limits_{[-h,0]}\|z_0^\prime\|^2.
			    		  					\end{array}
			    		  				$$
			    		  				Hence, the inequalities: \begin{equation}\label{qiqi}
			    		  					P\le \beta I,\;{p_1}\le \beta,\; {p_2}\le \beta
			    		  				\end{equation}
			    		  				guarantee that the trajectories $(X(t;X_0),z(x,t;z_0))$ starting from initial function $(X_0,z_0)\in \mathcal{X}_\beta$ remain within $\mathcal{X}_z$,
			    		  				where
			    		  				$$
			    		  					\mathcal{X}_z=\left\{(X,z)\in \mathcal H_1: X^\top PX+{p_1}\|z\|^2+{p_2}\|z_x\|^2\le 1\right\}.
			    		  				$$
			    		  				Note that the ellipsoid $\mathcal{X}_z$ is contained in $\mathcal{L}(c_1,c_2,c_3,\bar u)$, if the following implication holds
			    		  				$$
			    		  				\begin{array}{ll}
			    		  					X^\top PX+{p_1}\|z\|^2+{p_2}\|z_x\|^2\le 1 \Longrightarrow c_1^2|X|^2+c_2^2\|z\|^2+c_3^2\|z_x\|^2\le \dfrac{\bar{u}^2}{3}
			    		  					\end{array}
			    		  				$$
			    		  				for all $(X(t), z(x,t))$, i.e. if
			    		  				$$
			    		  				\begin{array}{ll}
			    		  					c_1^2|X|^2+c_2^2\|z\|^2+c_3^2\|z_x\|^2\le \dfrac{\bar{u}^2}{3}\left[X^\top PX+{p_1}\|z\|^2+{p_2}\|z_x\|^2\right].
			    		  					\end{array}
			    		  				$$
			    		  				The latter inequality is guaranteed if
			    		  				\begin{equation}\label{wwww}
			    		  					P\dfrac{\bar u^2}{3}-c_1^2I\geq 0,\; p_1\dfrac{\bar u^2}{3}-c_2^2\geq 0,\; p_2\dfrac{\bar u^2}{3}-c_3^2\geq 0.
			    		  				\end{equation}	
			    		  		Therefore, the LMIs \dref{wwww} guarantee the saturation avoidance, and together with Proposition \ref{nii}
			    		  		and the condition \dref{qiqi} imply that
			    		  		$$\lim\limits_{t\to\infty}\|(X(t;X_0),z(x,t;z_0))\|_{\mathcal H_1}=0.$$
			    		  		
			    		  		Returning to the original system by the transformation \dref{wei} and \dref{lyp}, we obtain that
			    		  		$$
			    		  		\begin{array}{ll}
			    		  		\|z_x\|&\le \|u_x\|+\max\limits_{0\le y\le x\le 1}|k_x(x,y)|\|u\|+\max\limits_{0\le x\le 1}|\gamma^\prime(x)||X|\\&+\left[\max\limits_{0\le x\le 1}|q(x,x)|+\max\limits_{0\le y\le x\le 1}|q_x(x,y)|\right]\|w\|.
			    		  		\end{array}
			    		  		$$
			    		  		It follows from \dref{pq} and \dref{pp} that
			    		  		$$
			    		  		\begin{array}{ll}
			    		  		|X|^2+\|z\|^2+\|z_x\|^2\le  M_1|X|^2+M_2\|u\|^2+4\|u_x\|^2,
			    		  		\end{array}
			    		  		$$
			    		  		where $$
			    		  		\begin{array}{ll}
			    		  		M_1&=\left\{8\left[\max\limits_{0\le x\le 1}|q(x,x)|+\max\limits_{0\le y\le x\le 1}|q_x(x,y)|\right]^2+2\left[1+\max\limits_{0\le y\le x\le 1}|q(x,y)|\right]^2\right\}\left[\max\limits_{0\le x\le 1}|\gamma(x)|\right]^2\\&+4\max\limits_{0\le x\le 1}|\gamma^\prime(x)|^2+1,
			    		  		\end{array}
			    		  		$$
			    		  		$$
			    		  		\begin{array}{ll}
			    		  		M_2&= 8\left[\max\limits_{0\le x\le 1}|q(x,x)|+\max\limits_{0\le y\le x\le 1}|q_x(x,y)|\right]^2\left[1+\max\limits_{0\le y\le x\le 1}|k(x,y)|\right]^2+4\max\limits_{0\le y\le x\le 1}|k_x(x,y)|^2\\
			    		  		&+2\left[1+\max\limits_{0\le y\le x\le 1}|k(x,y)|\right]^2\left[1+\max\limits_{0\le y\le x\le 1}|q(x,y)|\right]^2.
			    		  		\end{array}
			    		  		$$
			    		  		Denote
			    		  		$$
			    		  		\begin{array}{ll}
			    		  		\mathcal{X}_{u}=&\left\{(X_0,u_0)\in W_1:M_1\max\limits_{[-h,0]}|X_0|^2+M_2\max\limits_{[-h,0]}\|u_0\|^2
			    		  		+4\max\limits_{[-h,0]}\|u_0^\prime\|^2\le \beta^{-1}\right\}.
			    		  		\end{array}
			    		  		$$
			    		  		Then, we obtain the following result:
			    		  		\begin{theorem}\label{3.4}Given gains $K$ and $c$,  and tuning parameters $r>0$, $0<r_1<2$, $0<\delta_1=\delta_0$, let there exist an $n\times n$ matrix  $P>0$, and scalars $p_1>0$, $p_2>0$, $\lambda>0$ and $\lambda_1\geq0$ that satisfy the LMIs \dref{yyyy}, \dref{ppp}, \dref{ke}, \dref{qiqi} and \dref{wwww}. Then for all $h_0>0$ and $h>0$, the classical solutions of \dref{an} with Neumann boundary actuation \dref{as} starting from initial functions $(X_0,u_0)\in  \mathcal{X}_{u}$ converge to zero
			    		  			for all delays $\tau$ subject to \dref{ta}, i.e.
			    		  			$$
			    		  			\lim_{t\to \infty}\|(X(t;X_0),u(x,t;u_0))\|_{\mathcal H_1}=0.
			    		  			$$
			    		  		\end{theorem}
			    		  		
			    		  			\section{Examples}
			    		  				 	\begin{example}
			    		  				 		Consider the system \dref{an} with Dirichlet actuation, and
			    		  				 		the scalar $x(t)\in \mathbb{R}$ with
			    		  				 		$A=1$, $B=1$,  $A_1=0.4$, $a_2=0.1$, $a=0.2$, and $\bar u=20$. For the target system \dref{o}, we choose $K=-2$, $c=0.8$.
			    		  				 	    In order to enlarge the volume of the ellipse inside of the domain of attraction, we would like to minimize $\beta$.
			    		  				 		By Proposition \ref{ni}, with $\delta_0=\delta_1=0.3$, $c_1=0.91$, $c_2=2.93$, $r=1$, we obtain that $\min\beta=0.0739$, and the largest obtained ellipsoid inside of domain of attraction is given by
			    		  				 		$$ \mathcal{X}_\beta=\{(X_0,z_0)\in W:\max_{[-h,0]}|X_0|^2+\max_{[-h,0]}\|z_0\|^2\le 13.53\}.$$
			    		  				 		By Theorem \ref{3.2}, with $M_1=18.15$, $M_2=30.31$, we obtain
			    		  				 	$$
			    		  				 		\begin{array}{ll}
			    		  				 		\mathcal{X}_u=\{(X_0,u_0)\in W_1: 1.34\max\limits_{[-h,0]}|X_0|^2+2.24\max\limits_{[-h,0]}\|u_0\|^2
			    		  				 		\le 1\}.
			    		  				 		\end{array}
			    		  				 	$$
			    		  				 	
			    		  				 	\end{example}
			    		  				 	Next, a finite difference method is applied to compute
			    		  				 	the displacement of coupled heat and ODE
			    		  				 	system to illustrate the effect of the proposed
			    		  				 	feedback control law \dref{s}.
			    		  				 	The steps of space and time are taken as 0.04 and 0.0002,
			    		  				 	respectively.
			    		  				 	In Figure 1 and Figure 2, we choose the delay $\tau(t)\equiv h=0.4$. \\
			    		  				 	Figure 1 demonstrates the state $(X(t), u(x,t))$ of the closed-loop original system of \dref{an} with saturated control \dref{s}. We choose the initial conditions:
			    		  				 	$X(\theta)\equiv0.82$, $u(x,\theta)\equiv 0.29\cos (\pi x)$, $\theta\in [-0.4,0]$. Hence,
			    		  				 	$$1.34\max_{[-h,0]}|X_0|^2+2.24\max_{[-h,0]}\|u_0\|^2=0.99<1.$$
			    		  				 	It is seen that the initial values are chosen inside the ellipsoid $\mathcal{X}_u$.
			    		  				 	The results show that the states of ODE and heat PDE converge in Figure 1. \\
			    		  				 	Figure 2 illustrates instability for initial values taken outside $X_u$: $X(\theta)\equiv 5 $, $u(x,\theta)\equiv 4\cos (\pi x)$, $\theta\in [-0.4,0]$. Here, $$1.34\max_{[-h,0]}|X_0|^2+2.24\max_{[-h,0]}\|u_0\|^2=7.17^2>1.$$

			    		  				 		\begin{example}
			    		  				 			Consider the system \dref{an} with Neumann actuation, and
			    		  				 			the scalar $x(t)\in \mathbb{R}$ with
			    		  				 			$A=1$, $B=1$, $A_1=0.4$,  $a_2=0.1$, $a=0.2$, and $\bar u=50$. For the target system \dref{o}, we choose $K=-4$, $c=1.8$.
			    		  				 			In order to enlarge the volume of the ellipse inside of the domain of attraction, we would like to minimize $\beta$.
			    		  				 			By Proposition \ref{nii}, with $\delta_0=\delta_1=0.5$, $c_1=6.98$, $c_2=9.9$, $c_3=1$, $r=r_1=1$, we obtain that $\min\beta=0.1176$, and the largest obtained ball inside of domain of attraction is given by
			    		  				 			$$
			    		  				 			\begin{array}{ll}
			    		  				 				\mathcal{X}_\beta=\{(X_0,z_0)\in W:\max\limits_{[-h,0]}|X_0|^2+\max\limits_{[-h,0]}\|z_0\|^2+\max\limits_{[-h,0]}\|z_0^\prime\|^2\le 8.50\}.
			    		  				 			\end{array}
			    		  				 			$$
			    		  				 			By Theorem \ref{3.4}, with $M_1=118.7$, $M_2=141.8$, we obtain
			    		  				 		$$
			    		  				 			\begin{array}{ll}
			    		  				 		\mathcal{X}_{u}=\left\{(X_0,u_0)\in W_1:13.96\max\limits_{[-h,0]}|X_0|^2
			    		  				 		+16.67\max\limits_{[-h,0]}\|u_0\|^2+0.47\max\limits_{[-h,0]}\|u_0^\prime\|^2\le 1\right\}.
			    		  				 		\end{array}
			    		  				 		$$
			    		  				 		\end{example}
			    		  				 		Also in this case, we obtain that the simulations of the solutions confirm
			    		  				 		the theoretical results. Thus, starting inside the ellipsoid with initial conditions: $X(\theta)\equiv0.26$, $u(x,\theta)\equiv 0.05\cos (\pi x)$, $\theta\in [-0.4,0]$, the system is stable. However, starting outside the ellipsoid with initial conditions: $X(\theta)\equiv 3$, $u(x,\theta)\equiv 0.05\cos (\pi x)$, $\theta\in [-0.4,0]$, the system is unstable and the solution of the system becomes unbounded.
			    		  				 	\begin{figure}
			    		  				 		\begin{center}
			    		  				 		\subfigure[State $X(t)$]
			    		  				 		{\includegraphics[width=6cm,height=4.5cm]{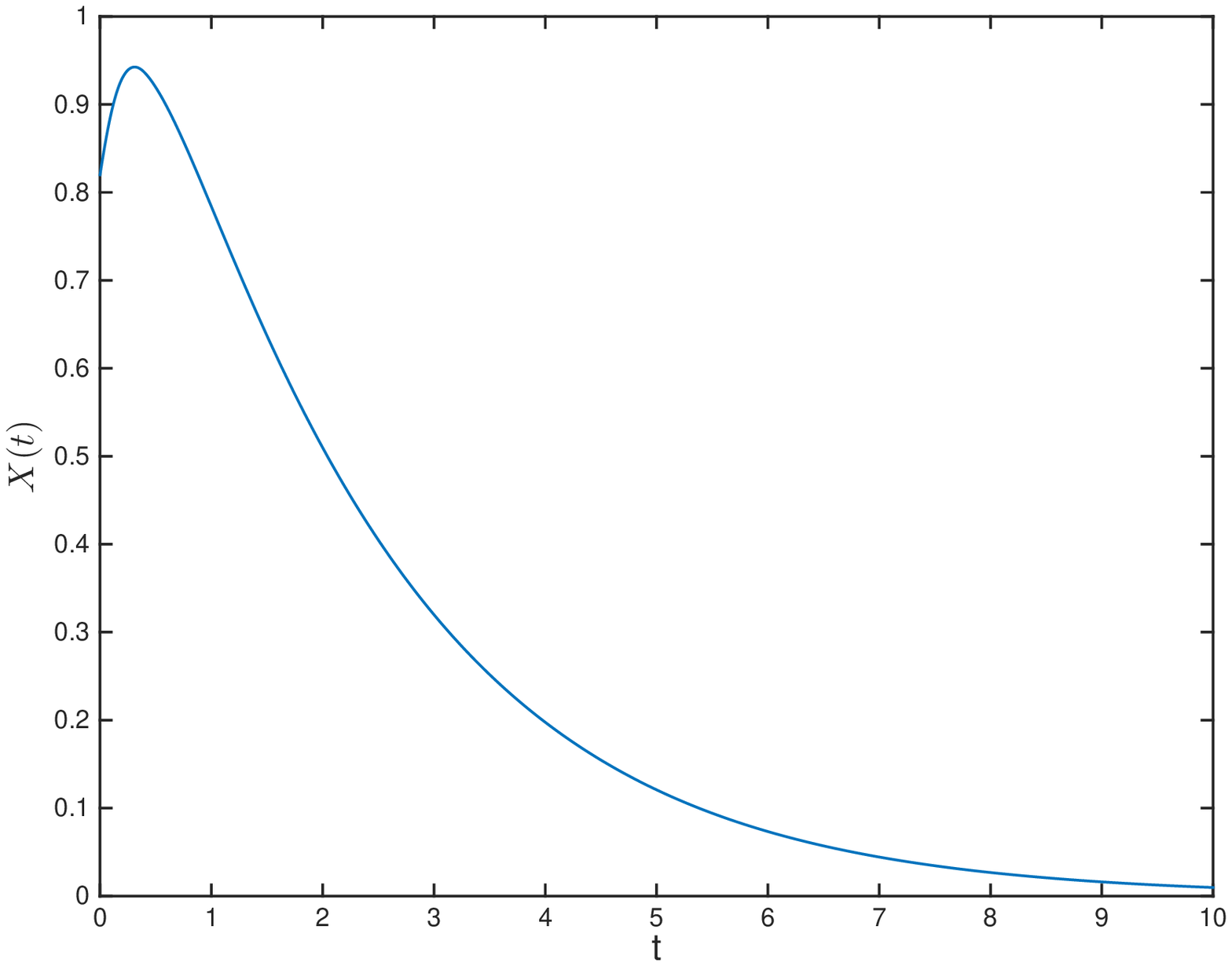}\label{Fig1.1}}
			    		  				 		\subfigure[State $u(x,t)$]
			    		  				 		{\includegraphics[width=6cm,height=4.5cm]{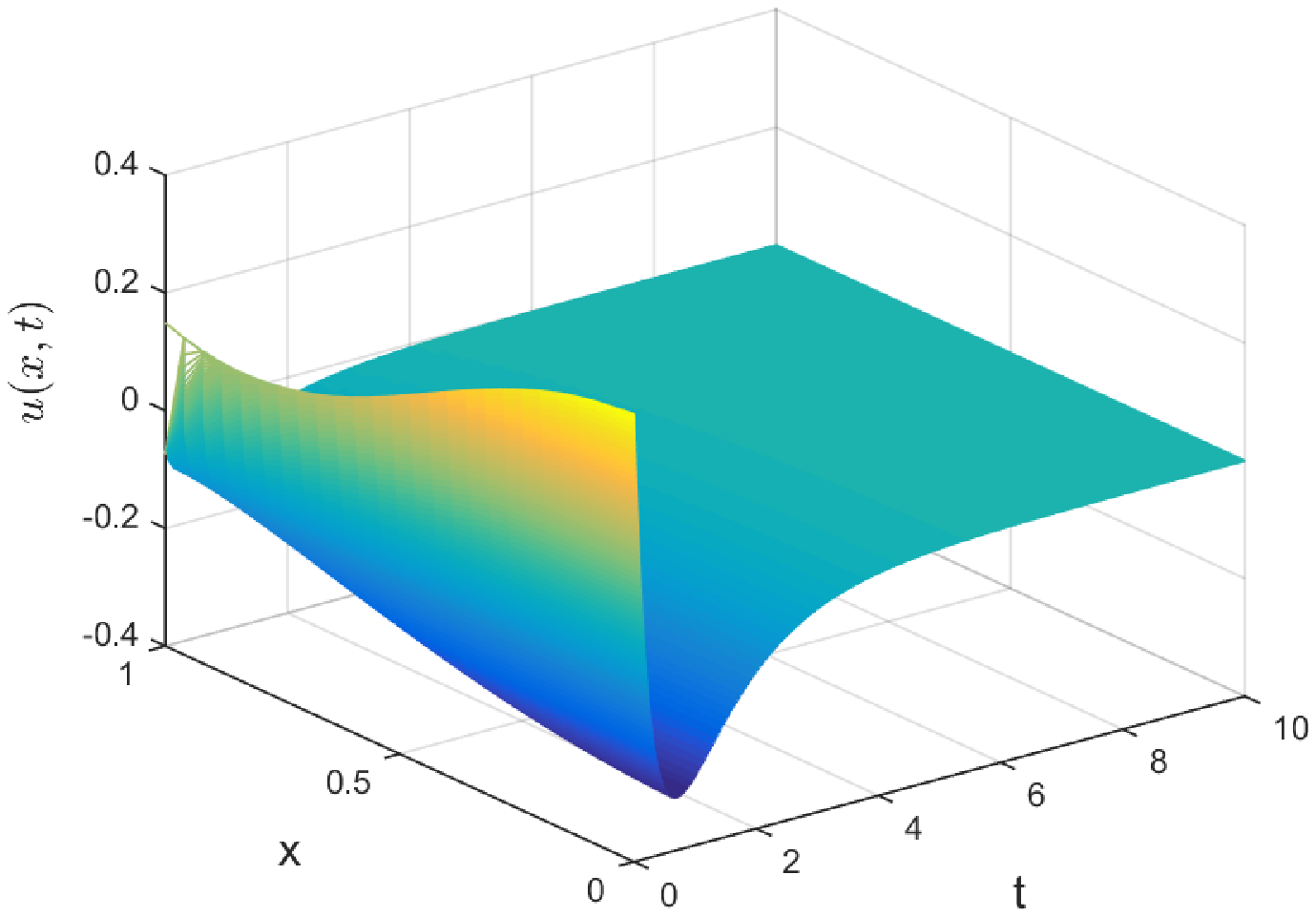}\label{Fig1.2}}
			    		  				 		\caption{State when the initial values are chosen inside the ellipsoid $\mathcal{X}_u$} \label{off}
			    		  				 		\end{center}
			    		  				 	\end{figure}

			    		  				 	\begin{figure}
			    		  				 		\begin{center}
			    		  				 		\subfigure[State $X(t)$]
			    		  				 		{\includegraphics[width=6cm,height=4.5cm]{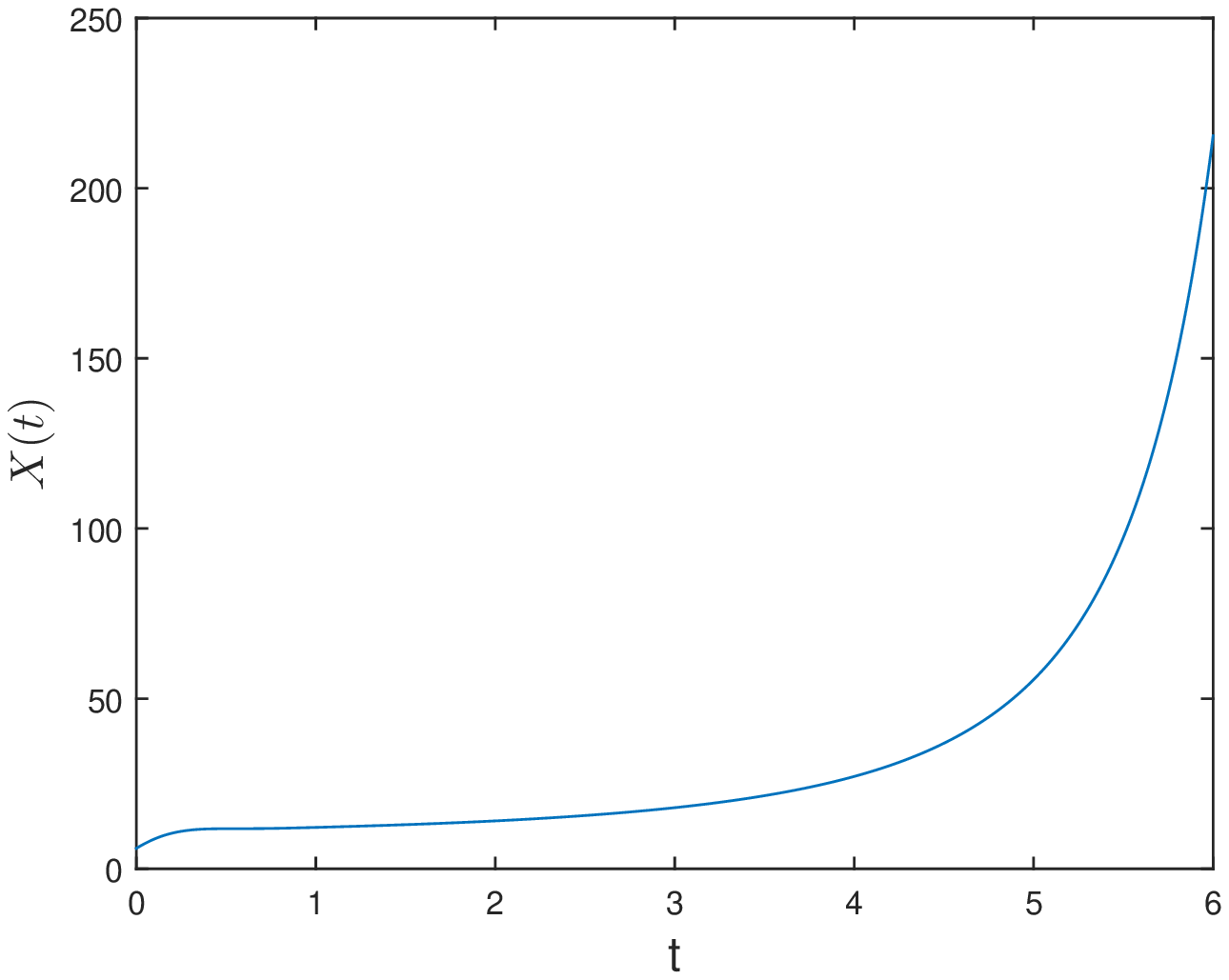}\label{Fig1.3}}
			    		  				 		\subfigure[State $u(x,t)$]
			    		  				 		{\includegraphics[width=6cm,height=4.5cm]{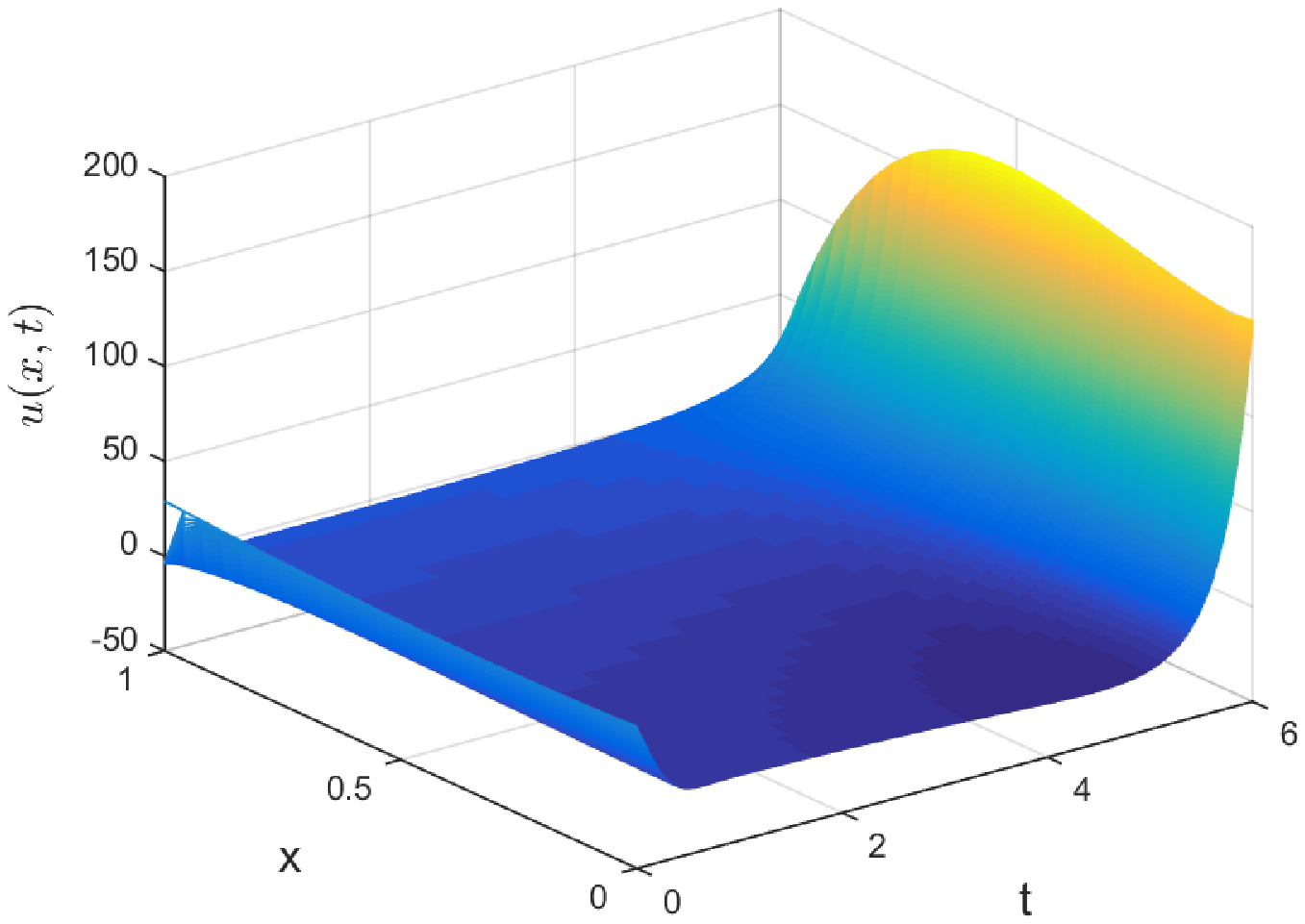}\label{Fig1.4}}
			    		  				 		\caption{State when the initial values are chosen outside the  ellipsoid $\mathcal{X}_u$} \label{on}
			    		  				 		\end{center}
			    		  				 	\end{figure}	

				\section{Conclusion}
				This paper for the first time studied boundary control of PDEs in the presence of saturation.
 Boundary stabilization of ODE-heat cascade with state time-varying delay was considered.
  The backstepping method was extended to  cascade of systems with state delays.
   An estimate on the domain of attraction in the presence of actuator saturation was found
   by using LMIs. Numerical examples illustrated the effectiveness of the proposed design method.

   The suggested approach may be extended to cascaded nonlinear ODE-Heat system, where the nonlinear term  satisfies the globally Lipschitz condition, and to observer-based control of such a system. The presented method gives efficient tools for various control problems for PDEs  with input constraints.  These may be the topics for future research.

			\end{document}